\documentclass{amsart}
\usepackage{latexsym,amsmath,amssymb}
\usepackage{cite}
\usepackage{color}
\newtheorem{thm}{Theorem}[section]
\newtheorem{cor}[thm]{Corollary}

\theoremstyle{definition}
\theoremstyle{remark}
\newtheorem{rem}[thm]{Remark}
\numberwithin{equation}{section}

\begin{document}

\title
{Nuclear Weighted composition operators between different $L^p$-spaces }

\author{\sc M. S. Al Ghafri, Y. Estaremi and S. Shamsigamchi }

\email{mohammed.alghafri@utas.edu.om}\email{y.estaremi@gu.ac.ir}\email{S.shamsi@pnu.ac.ir}
\address{Department of Mathematics, University of Technology and Applied Sciences, Rustaq {329},  Oman,}
\address{Department of Mathematics and Computer Sciences,
Golestan University, Gorgan, Iran,}
\address{Department of Mathematics, Payame Noor(PNU) university, Tehran, Iran.}

\thanks{}



\date{}

\dedicatory{}

\commby{}

\begin{abstract}
We provide complete characterisations of nuclear weighted composition operators between
two distinct $L^p(\mu)$-spaces, where $1\leq p<\infty$. As a consequence, when the underlying measure space is
non-atomic, the only nuclear weighted composition operator between $L^p(\mu)$-spaces is the zero operator.

\end{abstract}

\maketitle


\section{\textsc{\bf Introduction and Preliminaries}}

In this paper we are concerned with weighted composition operators between two distinct $L^p(\mu)$-spaces, that we recall the definition and some notations about it in the following.\\

Let $(X,\Sigma,\mu)$ be a $\sigma$-finite measure space. We adopt the following decomposition of $(X,\Sigma,\mu)$:
\[
X = \bigcup_{i=1}^{\infty} A_i \,\cup\, B,
\]
where $\{A_i\}_{i=1}^{\infty}$ is a countable collection of pairwise-disjoint atoms and
$B$, being disjoint from each $A_i$, is non-atomic. This decomposition is unique in the
sense that equality of two $\Sigma$-measurable sets is interpreted as their symmetric
difference having $\mu$-measure zero.

The $\sigma$-finiteness of $(X,\Sigma,\mu)$ ensures that $\mu(A_i) < \infty$ for every
$i \in \mathbb{N}$. Moreover, if $X = \bigcup_{i=1}^{\infty} A_i$ (respectively $X = B$),
then $(X,\Sigma,\mu)$ is said to be \emph{atomic} (respectively \emph{non-atomic}).
The limits in the theorems of this paper are assumed to take the value zero when the
number of atoms in $X$ is finite. The following facts will be useful.

 Let $E$ be a non-atomic set in $\Sigma$ with $\mu(E) > 0$.

 There is a sequence $\{E_n\}_{n=1}^{\infty}$ of pairwise-disjoint (or decreasing)
$\Sigma$-measurable subsets of $E$ with $\mu(E_n) > 0$ for each $n \in \mathbb{N}$ and
\[
\lim_{n \to \infty} \mu(E_n) = 0.
\]

 For every real number $\alpha$ satisfying $0 < \alpha < \mu(E)$, there is a set
$E_\alpha \in \Sigma$ with $E_\alpha \subset E$ and $\mu(E_\alpha) = \alpha$.

 A $\Sigma$-measurable function $f : X \to \mathbb{C}$ is constant
$\mu$-almost everywhere on an atom $M \in \Sigma$.
For $1 \le q < \infty$, the measure $\mu_q$ defined by
\[
\mu_q(E) := \int_{\varphi^{-1}(E)} |u|^q \, d\nu
\quad \text{for every } E \in \Sigma
\]
is absolutely continuous with respect to $\mu$. The corresponding Radon--Nikodym
derivative, denoted by $\left[\frac{d\mu_q}{d\mu}\right]$.\\

Let $(X, \mathcal{F}, \mu)$ be a measure space, $\varphi:X\rightarrow X$ be a non-singular measurable transformation (we mean $\mu\circ\varphi^{-1}\ll\mu$, where $\mu\circ\varphi^{-1}(A)=\mu(\varphi^{-1}(A)$) and $h=\frac{d\mu\circ\varphi^{-1}}{d\mu}$ (the Radon-Nikodym derivative of the measure $\mu\circ\varphi^{-1}$ with respect to the measure $\mu$) 
Let $\varphi^{-1}(\Sigma)$ denote the $\sigma$-algebra generated by $\{\varphi^{-1}(A): A\in \Sigma\}$. If $f$ is a non-negative $\Sigma$-measurable function on $X$, then there exists a unique (a.e.) $\varphi^{-1}(\Sigma)$-measurable function $E(f)$, called the conditional expectation 
of $f$ with respect to $\varphi^{-1}(\Sigma)$, such that 
$$\int_Afd\mu=\int_AE(f)d\mu, \ \ \ \ \ A\in \varphi^{-1}(\Sigma).$$ 
We shall also need the following facts:
\begin{enumerate}
  \item $f$ is $\varphi^{-1}(\Sigma)$-measurable if and only if $f\equiv g\circ \varphi$, for some $\Sigma$-measurable function $g$.
  \item $E(f.g\circ\varphi)=E(f).g\circ\varphi$, when the conditional expectations are well-defined.
  \end{enumerate}
  In view of $(1)$ above we write $E(f)\circ\varphi^{-1}$ to denote a function $g$ for which $E(f)\equiv g\circ \varphi$.\\

Let $M_u(f)=u.f$, $C_{\varphi}f=f\circ\varphi$, and let 
$$Wf=M_uC_{\varphi}f=uC_{\varphi}f=u.f\circ \varphi,$$
for every $f\in L^p(\mu)$, be the multiplication operator, composition operator and weighted composition operator, respectively, on $L^p(\mu)$, in which $u:X\rightarrow \mathbb{C}$ is a measurable function. It is known that if $W$ is bounded, then $\|W\|=\|J_p\|_{\infty}$, in which $J_p=hE(|u|^p)\circ\varphi^{-1}$. For more information about composition and weighted composition operators on $L^2(\mu)$, one can see \cite{bbl,cj}.\\
By these observation we have the following important relation.\\
If $uC_\varphi : L^p(\mu) \to L^q(\nu)$, then
\[
\|u C_\varphi f\|_{L^q(\nu)}^q
= \int_X \frac{d\mu_q}{d\mu} |f|^q \, d\mu,
\quad \text{for every } f \in L^p(\mu).
\]
Hence we get that $J_p=\frac{d\mu_p}{d\mu}$, $\mu$.a.e.\\
Let $X$ and $Y$ be Banach spaces and $X^*$, $Y^*$ be their Duals(the set of all continuous or (equivalently) bounded linear functionals on $X$ and $Y$ with the usual norm), respectively. A linear operator $T:X\rightarrow Y$ is called nuclear if there are two sequences $\{f_n\}_{n\in \mathbb{N}}\subseteq X^*$ and $\{y_n\}_{n\in \mathbb{N}}\subseteq Y$ such that $\sum^{\infty}_{n=1}\|f_n\|_{X^*}\|y_n\|_{Y}<\infty$ and
\begin{equation}\label{A}
Tx=\sum^{\infty}_{n=1}f_n\otimes y_n(x)=\sum^{\infty}_{n=1}f_n(x)y_n, \ \ \ \ x\in X.
\end{equation}
Define
$$\|T\|_N:=\inf\{\sum^{\infty}_{n=1}\|f_n\|_{X^*}\|y_n\|_Y\},$$
in which the infimum is taken over all representations of $T$ in the form \ref{A}. The norm $\|.\|_N$ is known as nuclear norm and it's clear that if $T$ is nuclear, then it's nuclear norm is finite and $\|T\|\leq \|T\|_N$.

It is easy to check that the operator defined as
$$f_n\otimes y_n:X\rightarrow Y, \ \ \ f_n\otimes y_n(x)=f_n(x)y_n,$$
is a rank one operator. Since finite rank operators are compact and the series $T=\sum^{\infty}_{n=1}f_n\otimes y_n$ is absolutely convergent, then nuclear operators are compact. It is a direct consequence from the definition of nuclear operators that if $S, T:X\rightarrow X$ are nuclear, then $S+T$ is also nuclear.\\
If $S, T:X\rightarrow X$ are linear operators on the Banach spaces $X$ and $T$ is nuclear, the $ST$ and $TS$ are also nuclear, because
$$TS=\sum^{\infty}_{n=1}g_n\otimes y_n, \ \ \ \ ST=\sum^{\infty}_{n=1}f_n\otimes z_n, \ \ g_n=x^*_n\circ S,  \ \ z_n=S(y_n),$$
and also
$$\sum^{\infty}_{n=1}\|g_n\|\|y_n\|\leq \|S\|\sum^{\infty}_{n=1}\|f_n\|\|y_n\|<\infty, \ \ \sum^{\infty}_{n=1}\|f_n\|\|z_n\|\leq \|S\|\sum^{\infty}_{n=1}\|f_n\|\|y_n\|<\infty.$$
This implies that the set of all nuclear operators on the Banach space $X$ is a two sided ideal in the space of all linear operators on $X$.

Also, the linear operator $T$ is said to be absolutely summing there exists $M>0$ such that
$$\sum^{n}_{i=1}\|T(x_i)\|\leq M \sup_{f\in X^*, \|f\|\leq 1}|f(x_i)|=M\sup_{|\alpha_i|=1}\|\sum^n_{i=1}\alpha_ix_i\|,$$
for all finite sequences $\{x_i\}^n_{i=1}\subseteq X$.\\

Many operator-theoretic properties of weighted composition operators on $L^p$-spaces were studied in \cite{bbl,cce,cce1,chan,ll,ll1} and such operators acting between distinct $L^p$-spaces in \cite{jab,jp,ku,tak}. 
In this paper we characterize nuclear weighted composition operators between different $L^p(\mu)$-spaces. Consequently, we characterize nuclear composition and multiplication operators between distinct $L^p(\mu)$-spaces. 

\section{\textsc{\bf Nuclear weighted composition operators}}
In this section first we recall the Pietsch's domination theorem that will be used in the proof of our main results. Then we begin to investigate nuclearity of weighted composition operator $W=uC_{\varphi}:L^p(\mu)\rightarrow L^q(\mu)$ in three cases $p=q, p>q, p<q$, respectively. Also, we will have some results on nuclearity of multiplication and composition operators between two $L^p$-spaces.

\begin{thm}[Pietsch domination theorem {\cite[Theorem~2.12]{dj}}]\label{thm:Pietsch}
Let $1\le p<\infty$ and let $T:X\to Y$ be absolutely $p$--summing.
Then there exist $C>0$ and a Borel probability measure $\mu$ on the weak$^{*}$--compact unit ball $B_{X^*}$ such that
\[
\|Tx\|
\le C\left(\int_{B_{X^*}} |x^*(x)|^p\,d\mu(x^*)\right)^{1/p},
\qquad x\in X.
\]
In particular, for $p=1$,
\[
\|Tx\|\le C\int_{B_{X^*}} |x^*(x)|\,d\mu(x^*).
\]
\end{thm}
Now in the following we provide an equivalent conditions for nuclearity of weighted composition operators from $L^p(\mu)$ into itself. 

\begin{thm}\label{n.pp}
Let $(X, \mu, \Sigma)$ is a $\sigma$-finite measure spaces and $X=\cup^{\infty}_{n=1}A_n\cup B$, in which $A_n$'s are disjoint $\sigma$-atoms and $B$ is the non-atomic part. If $1\leq p<\infty$  and $W=uC_{\varphi}:L^p(\mu)\rightarrow L^p(\mu)$ is the weighted composition operator, then $W$ is nuclear if and only if $\mu(\varphi^{-1}(B)\cap S(u))=0$ and $\sum^{\infty}_{n=1}J_p(A_n)^{\frac{1}{p}}<\infty$.
\end{thm}
\begin{proof}
By assumptions we have $X=\cup^{\infty}_{n=1}A_n\cup B=\cup^{\infty}_{n=1}\varphi^{-1}(A_n)\cup \varphi^{-1}(B)$. Let $\mu(\varphi^{-1}(B)\cap S(u))=0$ and $\sum^{\infty}_{n=1}J_p(A_n)^{\frac{1}{p}}<\infty$. Then for every $f\in L^p(\mu)$,
\begin{align*}
u.f\circ \varphi&=u.\sum^{\infty}_{n=1}f\circ\varphi.\chi_{\varphi^{-1}(A_n)}\\
&=u.\sum^{\infty}_{n=1}(f.\chi_{A_n})\circ\varphi\\
&=u.\sum^{\infty}_{n=1}f(A_n).\chi_{A_n}\circ\varphi\\
&=\sum^{\infty}_{n=1}\frac{u.\chi_{\varphi^{-1}(A_n)}}{\mu(A_n)}\int_{A_n}fd\mu\\
&=\sum^{\infty}_{n=1}\phi_n(f)g_n,
\end{align*}
in which $g_n=\frac{u.\chi_{\varphi^{-1}(A_n)}}{\mu(A_n)}$ and $\phi_n(f)=\int_{A_n}fd\mu$. It is easy to see that $g_n\in L^p(\mu)$, $\phi_n\in (L^p(\mu))^*$ and $\|g_n\|_p=\frac{(J_p(A_n))^{\frac{1}{p}}}{\mu(A_n)^{\frac{1}{p'}}}$, $\|\phi_n\|\leq\mu(A_n)^{\frac{1}{p'}}$ (for $1<p<\infty$) and also $\|g_n\|_1=J_1(A_n)$ and $\|\phi_n\|\leq1$ (for $p=1$), in which $\frac{1}{p}+\frac{1}{p'}=1$. This implies that 
$$\sum^{\infty}_{n=1}\|\phi_n\|.\|g_n\|_p\leq \sum^{\infty}_{n=1}J_p(A_n)^{\frac{1}{p}}<\infty.$$
So $W=u.C_{\varphi}$ is nuclear.\\
Conversely, let $W=u.C_{\varphi}$ be nuclear. Then it is absolutely summing and compact, so by Theorem \ref{thm:Pietsch}, there is a Borel probability measure $\nu$ on the weak$^*$ compact closed unit ball $B_{p'}$ of $(L^p(\mu))^*\simeq L^{p'}(\mu)$ (case $1<p<\infty$), $B_{\infty}$ of $(L^1(\mu))^*\simeq L^{\infty}(\mu)$ (case $p=1$) and $C>0$ such that
$$\|u.C_{\varphi}f\|_p\leq C \int_{B_{p'}}|\xi(f)|\nu(\xi), \ \ \ \ \ \ f\in L^p(\mu),$$
and 
$$\|u.C_{\varphi}f\|_1\leq C \int_{B_{\infty}}|\xi(f)|\nu(\xi), \ \ \ \ \ \ f\in L^1(\mu).$$
For each $n\in \mathbb{N}$, we set $f_n=\frac{\chi_{A_n}}{\mu(A_n)^{\frac{1}{p}}}$ and then we have
\begin{align*}
J_p(A_n)^{\frac{1}{p}}&=\frac{1}{\mu(A_n)^{\frac{1}{p}}}\left(\int_X|u|^p\chi_{A_n}\circ\varphi d\mu\right)^{\frac{1}{p}}\\
&=\|u.f_n\circ\varphi\|_p\\
&\leq C\int_{B_{p'}}|\xi(f_n)|\nu(\xi)\\
&=C\int_{B_{p'}}|\int_X f_n.g_{\xi}d\mu\nu(g_{\xi})|\\
&\leq C \int_{B_{p'}}\|f_n\|_p.\|g_{\xi}.\chi_{A_n}\|_{p'}\nu(g_{\xi})\\
&\leq C\int_{B_{p'}}\|g_{\xi}.\chi_{A_n}\|_{p'}\nu(g_{\xi}),\\
\end{align*}
and similarly $J_1(A_n)\leq C\int_{B_{\infty}}\|g_{\xi}.\chi_{A_n}\|_{\infty}\nu(g_{\xi})$. On the other hands compactness of $W$ implies that $\mu(\varphi^{-1}(B)\cap S(u))=0$. So $\|g_{\xi}\|_{p'}=\sum^{\infty}_{n=1}\|g_{\xi}.\chi_{A_n}\|_{p'}$ and $\|g_{\xi}\|_{\infty}=\sum^{\infty}_{n=1}\|g_{\xi}.\chi_{A_n}\|_{\infty}$. Therefore
\begin{align*}
\sum^{\infty}_{n=1}J_p(A_n)^{\frac{1}{p}}&\leq C\sum^{\infty}_{n=1}\int_{B_{p'}}\|g_{\xi}.\chi_{A_n}\|_{p'}\nu(g_{\xi})\\
&\leq C\int_{B_{p'}}\sum^{\infty}_{n=1}\|g_{\xi}.\chi_{A_n}\|_{p'}\nu(g_{\xi})\\
&=C\int_{B_{p'}}\|g_{\xi}\|_{p'}\nu(g_{\xi})\\
&\leq C\nu(B_{p'})\leq C,
\end{align*}
and similarly $\sum^{\infty}_{n=1}J_1(A_n)\leq C$.
Thus we get the result.
\end{proof}
Here we have some corollaries for nuclear multiplication and composition operators on $L^p$ spaces.

\begin{cor}
Let $(X, \mu, \Sigma)$ is a $\sigma$-finite measure spaces and $X=\cup^{\infty}_{n=1}A_n\cup B$, in which $A_n$'s are disjoint $\sigma$-atoms and $B$ is the non-atomic part. If $1\leq p<\infty$  and $C_{\varphi}:L^p(\mu)\rightarrow L^p(\mu)$ is the composition operator, then $C_{\varphi}$ is nuclear if and only if $\mu(\varphi^{-1}(B))=0$ and $\sum^{\infty}_{n=1}h(A_n)^{\frac{1}{p}}<\infty$.
\end{cor}
\begin{cor}
Let $(X, \mu, \Sigma)$ is a $\sigma$-finite measure spaces and $X=\cup^{\infty}_{n=1}A_n\cup B$, in which $A_n$'s are disjoint $\sigma$-atoms and $B$ is the non-atomic part. If $1\leq p<\infty$  and $M_u:L^p(\mu)\rightarrow L^p(\mu)$ is the multiplication operator, then $M_u$ is nuclear if and only if $\mu(B\cap S(u))=0$ and $\sum^{\infty}_{n=1}u(A_n)<\infty$.
\end{cor}
In the next theorem we characterize nuclear weighted composition operators for the case $\infty>p>q>1$.
\begin{thm}\label{n.pq}
Let $(X, \mu, \Sigma)$ is a $\sigma$-finite measure spaces and $X=\cup^{\infty}_{n=1}A_n\cup B$, in which $A_n$'s are disjoint $\sigma$-atoms and $B$ is the non-atomic part. If $1\leq q<p<\infty$  and $W=uC_{\varphi}:L^p(\mu)\rightarrow L^q(\mu)$ is the weighted composition operator, then $W$ is nuclear if and only if $\mu(\varphi^{-1}(B)\cap S(u))=0$ and $\sum^{\infty}_{n=1}J_q(A_n)^{\frac{1}{q}}\mu(A_n)^{\frac{1}{r}}<\infty$, in which $\frac{1}{p}+\frac{1}{r}=\frac{1}{q}$.
\end{thm}
\begin{proof}
Similar to the proof of Theorem \ref{n.pp}, for every $f\in L^p(\mu)$ we can write
$$W(f)=u.f\circ \varphi=\sum^{\infty}_{n=1}\phi_n(f)g_n,$$

in which $g_n=\frac{u.\chi_{\varphi^{-1}(A_n)}}{\mu(A_n)}$ and $\phi_n(f)=\int_{A_n}fd\mu$. As we considered in the proof of \ref{n.pp}, we can divide the procedure of the proof into two cases $q=1$ and $q>1$ and complete the proof. We assume that $\frac{1}{\infty}=0$. Hence we do the proof for both cases jointly. It is easy to see that $g_n\in L^q(\mu)$, $\phi_n\in (L^p(\mu))^*$ and $\|g_n\|_q=\frac{(J_q(A_n))^{\frac{1}{q}}}{\mu(A_n)^{\frac{1}{q'}}}$ and $\|\phi_n\|\leq\mu(A_n)^{\frac{1}{p'}}$, where $\frac{1}{p}+\frac{1}{p'}=1$, $\frac{1}{q}+\frac{1}{q'}=1$ and $\frac{1}{r'}+\frac{1}{r}=1$. Since $\frac{1}{p}+\frac{1}{r}=\frac{1}{q}$, then easily we get that $\frac{1}{p'}+\frac{1}{r'}=1+\frac{1}{q'}$. By these observations we get that 
$$\sum^{\infty}_{n=1}\|\phi_n\|.\|g_n\|_q\leq \sum^{\infty}_{n=1}J_q(A_n)^{\frac{1}{q}}\mu(A_n)^{\frac{1}{p'}-\frac{1}{q'}}=\sum^{\infty}_{n=1}J_q(A_n)^{\frac{1}{q}}\mu(A_n)^{\frac{1}{r}}<\infty.$$
Therefore $W=u.C_{\varphi}$ is nuclear.\\
Conversely, let $W=u.C_{\varphi}$ be nuclear. Then it is absolutely summing and compact, so by Theorem \ref{thm:Pietsch}, there is a Borel probability measure $\nu$ on the weak$^*$ compact closed unit ball $B_{p'}$ of $(L^p(\mu))^*\simeq L^{p'}(\mu)$ and $C>0$ such that
$$\|u.C_{\varphi}f\|_q\leq C \int_{B_{p'}}|\xi(f)|\nu(\xi), \ \ \ \ \ \ f\in L^p(\mu).$$
For each $n\in \mathbb{N}$, we set $f_n=\chi_{A_n}$ and then we have
\begin{align*}
J_q(A_n)^{\frac{1}{q}}\mu(A_n)^{\frac{1}{q}}&=\left(\int_X|u|^q\chi_{A_n}\circ\varphi d\mu\right)^{\frac{1}{q}}\\
&=\|u.f_n\circ\varphi\|_q\\
&\leq C\int_{B_{p'}}|\xi(f_n)|\nu(\xi)\\
&=C\int_{B_{p'}}|\int_X f_n.g_{\xi}d\mu\nu(g_{\xi})|\\
&\leq C \int_{B_{p'}}\|f_n\|_p.\|g_{\xi}.\chi_{A_n}\|_{p'}\nu(g_{\xi})\\
&\leq C\mu(A_n)^{\frac{1}{p}}\int_{B_{p'}}\|g_{\xi}.\chi_{A_n}\|_{p'}\nu(g_{\xi}).\\
\end{align*}
Hence 
$$J_q(A_n)^{\frac{1}{q}}\mu(A_n)^{\frac{1}{r}}=\frac{J_q(A_n)^{\frac{1}{q}}\mu(A_n)^{\frac{1}{q}}}{\mu(A_n)^{\frac{1}{p}}}\leq C\int_{B_{p'}}\|g_{\xi}.\chi_{A_n}\|_{p'}\nu(g_{\xi}).$$
 On the other hands compactness of $W$ implies that $\mu(\varphi^{-1}(B)\cap S(u))=0$. So $\|g_{\xi}\|_{p'}=\sum^{\infty}_{n=1}\|g_{\xi}.\chi_{A_n}\|_{p'}$. Therefore
\begin{align*}
\sum^{\infty}_{n=1}J_q(A_n)^{\frac{1}{q}}\mu(A_n)^{\frac{1}{r}}&\leq C\sum^{\infty}_{n=1}\int_{B_{p'}}\|g_{\xi}.\chi_{A_n}\|_{p'}\nu(g_{\xi})\\
&\leq C\int_{B_{p'}}\sum^{\infty}_{n=1}\|g_{\xi}.\chi_{A_n}\|_{p'}\nu(g_{\xi})\\
&=C\int_{B_{p'}}\|g_{\xi}\|_{p'}\nu(g_{\xi})\\
&\leq C\nu(B_{p'})\leq C.
\end{align*}
Thus we get the result.
\end{proof}
Here we have some corollaries for nuclear multiplication and composition operators between distinct $L^p$ spaces in the case $p>q$.
\begin{cor}
If $1\leq q<p<\infty$  and $M_u:L^p(\mu)\rightarrow L^q(\mu)$ is the multiplication operator, then $M_u$ is nuclear if and only if $\mu(B\cap S(u))=0$ and $\sum^{\infty}_{n=1}u(A_n)\mu(A_n)^{\frac{1}{r}}<\infty$, in which $\frac{1}{p}+\frac{1}{r}=\frac{1}{q}$.
\end{cor}
\begin{cor}
If $1\leq q<p<\infty$  and $C_{\varphi}:L^p(\mu)\rightarrow L^q(\mu)$ is the composition operator, then $C_{\varphi}$ is nuclear if and only if $\mu(\varphi^{-1}(B))=0$ and $\sum^{\infty}_{n=1}h(A_n)^{\frac{1}{q}}\mu(A_n)^{\frac{1}{r}}<\infty$, in which $\frac{1}{p}+\frac{1}{r}=\frac{1}{q}$.
\end{cor}
Now in the following we characterize nuclear weighted composition operators for the case $\infty>q>p>1$.
\begin{thm}\label{n.qp}
Let $(X, \mu, \Sigma)$ is a $\sigma$-finite measure spaces and $X=\cup^{\infty}_{n=1}A_n\cup B$, in which $A_n$'s are disjoint $\sigma$-atoms and $B$ is the non-atomic part. If $1\leq p<q<\infty$  and $W=uC_{\varphi}:L^p(\mu)\rightarrow L^q(\mu)$ is the weighted composition operator, then $W$ is nuclear if and only if $\mu(\varphi^{-1}(B)\cap S(u))=0$ and $\sum^{\infty}_{n=1}\frac{J_q(A_n)^{\frac{1}{q}}}{\mu(A_n)^{\frac{1}{h}}}<\infty$, in which $\frac{1}{q}+\frac{1}{h}=\frac{1}{p}$.
\end{thm}
\begin{proof}
Similar to the proof of Theorem \ref{n.pq}, for every $f\in L^p(\mu)$ we can write
$$W(f)=u.f\circ \varphi=\sum^{\infty}_{n=1}\phi_n(f)g_n,$$

in which $g_n=\frac{u.\chi_{\varphi^{-1}(A_n)}}{\mu(A_n)}$ and $\phi_n(f)=\int_{A_n}fd\mu$. It is easy to see that $g_n\in L^q(\mu)$, $\phi_n\in (L^p(\mu))^*$ and $\|g_n\|_q=\frac{(J_q(A_n))^{\frac{1}{q}}}{\mu(A_n)^{\frac{1}{q'}}}$ and $\|\phi_n\|\leq\mu(A_n)^{\frac{1}{p'}}$, where $\frac{1}{q}+\frac{1}{h}=\frac{1}{p}$. Since $\frac{1}{q}+\frac{1}{h}=\frac{1}{p}$, then easily we get that $\frac{1}{q'}+\frac{1}{h'}=1+\frac{1}{p'}$. By these observations we get that 
$$\sum^{\infty}_{n=1}\|\phi_n\|.\|g_n\|_q\leq \sum^{\infty}_{n=1}J_q(A_n)^{\frac{1}{q}}\mu(A_n)^{\frac{1}{p'}-\frac{1}{q'}}=\sum^{\infty}_{n=1}\frac{J_q(A_n)^{\frac{1}{q}}}{\mu(A_n)^{\frac{1}{h}}}<\infty.$$
Therefore $W=u.C_{\varphi}$ is nuclear.\\
Conversely, let $W=u.C_{\varphi}$ be nuclear. Then exactly the same as the proof of theorem \ref{n.pq}, by Theorem \ref{thm:Pietsch}, there is a Borel probability measure $\nu$ on the weak$^*$ compact closed unit ball $B_{p'}$ of $(L^p(\mu))^*\simeq L^{p'}(\mu)$ and $C>0$ such that
$$\|u.C_{\varphi}f\|_q\leq C \int_{B_{p'}}|\xi(f)|\nu(\xi), \ \ \ \ \ \ f\in L^p(\mu).$$
For each $n\in \mathbb{N}$, we set $f_n=\chi_{A_n}$ and then we have
$$
J_q(A_n)^{\frac{1}{q}}\mu(A_n)^{\frac{1}{q}}\leq C\mu(A_n)^{\frac{1}{p}}\int_{B_{p'}}\|g_{\xi}.\chi_{A_n}\|_{p'}\nu(g_{\xi}).\\
$$
Hence 
$$\frac{J_q(A_n)^{\frac{1}{q}}}{\mu(A_n)^{\frac{1}{h}}}=\frac{J_q(A_n)^{\frac{1}{q}}\mu(A_n)^{\frac{1}{q}}}{\mu(A_n)^{\frac{1}{p}}}\leq C\int_{B_{p'}}\|g_{\xi}.\chi_{A_n}\|_{p'}\nu(g_{\xi}).$$
 On the other hands compactness of $W$ implies that $\mu(\varphi^{-1}(B)\cap S(u))=0$. So $\|g_{\xi}\|_{p'}=\sum^{\infty}_{n=1}\|g_{\xi}.\chi_{A_n}\|_{p'}$. Therefore
\begin{align*}
\sum^{\infty}_{n=1}\frac{J_q(A_n)^{\frac{1}{q}}}{\mu(A_n)^{\frac{1}{h}}}&\leq C\sum^{\infty}_{n=1}\int_{B_{p'}}\|g_{\xi}.\chi_{A_n}\|_{p'}\nu(g_{\xi})\\
&\leq C\int_{B_{p'}}\sum^{\infty}_{n=1}\|g_{\xi}.\chi_{A_n}\|_{p'}\nu(g_{\xi})\\
&\leq C.
\end{align*}
Thus we get the result.
\end{proof}
Here we have some corollaries for nuclear multiplication and composition operators between distinct $L^p$ spaces in the case $q>p$.
\begin{cor}
 If $1\leq p<q<\infty$  and $M_u:L^p(\mu)\rightarrow L^q(\mu)$ is the multiplication operator, then $M_u$ is nuclear if and only if $\mu(B\cap S(u))=0$ and $\sum^{\infty}_{n=1}\frac{u(A_n)}{\mu(A_n)^{\frac{1}{h}}}<\infty$, in which $\frac{1}{q}+\frac{1}{h}=\frac{1}{p}$.
\end{cor}
\begin{cor}
If $1\leq p<q<\infty$  and $C_{\varphi}:L^p(\mu)\rightarrow L^q(\mu)$ is the composition operator, then $C_{\varphi}$ is nuclear if and only if $\mu(\varphi^{-1}(B)=0$ and $\sum^{\infty}_{n=1}\frac{h(A_n)^{\frac{1}{q}}}{\mu(A_n)^{\frac{1}{h}}}<\infty$, in which $\frac{1}{q}+\frac{1}{h}=\frac{1}{p}$.
\end{cor}
\begin{rem}
Let $uC_{\varphi}:L^p(\mu)\rightarrow L^q(\mu)$ be nuclear. Then the followings hold:
\begin{itemize}
  \item  If $1\leq p<q<\infty$, then $\mu(\varphi^{-1}(B)\cap S(u))=0$ and $\sum^{\infty}_{n=1}\frac{J_q(A_n)^{\frac{1}{q}}}{\mu(A_n)^{\frac{1}{h}}}<\infty$, in which $\frac{1}{q}+\frac{1}{h}=\frac{1}{p}$. Hence
      $$\lim_{n\rightarrow\infty}\left(\frac{J_q(A_n)}{\mu(A_n)^{\frac{q}{h}}}\right)^{\frac{1}{q}}=\lim_{n\rightarrow\infty}\frac{J_q(A_n)^{\frac{1}{q}}}{\mu(A_n)^{\frac{1}{h}}}=0.$$ Thus  $\lim_{n\rightarrow\infty}\frac{J(A_n)}{\mu(A_n)^{\frac{q}{h}}}=0$ and so $uC_{\varphi}$ is compact.
  \item  If $1\leqq<p<\infty$, then $\mu(\varphi^{-1}(B)\cap S(u))=0$ and $\sum^{\infty}_{n=1}J_q(A_n)^{\frac{1}{q}}\mu(A_n)^{\frac{1}{r}}<\infty$, in which $\frac{1}{p}+\frac{1}{r}=\frac{1}{q}$. Since
       $$\sum^{\infty}_{n=1}J_q(A_n)^{\frac{r}{q}}\mu(A_n)=\sum^{\infty}_{n=1}\left(J_q(A_n)^{\frac{1}{q}}\mu(A_n)^{\frac{1}{r}}\right)^r\leq\left(\sum^{\infty}_{n=1}J(A_n)^{\frac{1}{q}}\mu(A_n)^{\frac{1}{r}}\right)^r<\infty,$$
       then $\sum^{\infty}_{n=1}J_q(A_n)^{\frac{r}{q}}\mu(A_n)<\infty$. This implies that $uC_{\varphi}$ is compact.          
  \item  If $1\leq p=q<\infty$, then $\mu(\varphi^{-1}(B)\cap S(u))=0$ and $\sum^{\infty}_{n=1}J_q(A_n)^{\frac{1}{p}}<\infty$. Hence $$\left(\lim_{n\rightarrow\infty}J_p(A_n)\right)^{\frac{1}{p}}=\lim_{n\rightarrow\infty}J_p(A_n)^{\frac{1}{p}}=0$$
       and so $\lim_{n\rightarrow\infty}J_p(A_n)=0$. Therefore $uC_{\varphi}$ is compact.
\end{itemize}
\end{rem}

\end{document}